\documentclass[12pt]{article}
\usepackage{amssymb,amsmath,textcomp, amsthm}
\usepackage{cases}
\newtheorem{theorem}{Theorem}[section]

\newtheorem{lemma}[theorem]{Lemma}

\theoremstyle{definition}

\newtheorem{remark}{Remark}

\newcommand{\abs}[1]{\left\vert#1\right\vert}
\newcommand{\RR}{\mathbb{R}}

\newcommand{\norm}[1]{\left\Vert#1\right\Vert}

\begin{document}
\title{Existence of solutions for a $k$-Hessian equation and its connection with self-similar solutions}

\author{Justino S\'{a}nchez}
\date{}
\maketitle
\begin{center}
Departamento de Matem\'{a}ticas, Universidad de La Serena\\
 Avenida Cisternas 1200, La Serena, Chile.
\\email: jsanchez@userena.cl
\end{center}

\begin{abstract}
 Let $\alpha,\beta$ be real parameters and let $a>0$. We study radially symmetric solutions of 
\begin{equation*}
S_k(D^2v)+\alpha v+\beta \xi\cdot\nabla v=0,\, v>0\;\; \mbox{in}\;\; \mathbb{R}^n,\; v(0)=a,
\end{equation*}
where $S_k(D^2v)$ denotes the $k$-Hessian operator of $v$. For $\alpha\leq\frac{\beta(n-2k)}{k}\;\;\mbox{and}\;\;\beta>0$, we prove the existence of a unique solution to this problem, without using the phase plane method. We also prove existence and properties of the solutions of the above equation for other ranges of the parameters $\alpha$ and $\beta$. These results are then applied to construct different types of explicit solutions, in self-similar forms, to a related evolution equation. In particular, for the heat equation, we have found a new family of self-similar solutions of type II which blows up in finite time. These solutions are represented as a power series, called the Kummer function.
\end{abstract}

\section{Introduction} 
We briefly introduce the class of operators studied in this paper. For a twice-differentiable function $u$ defined on a domain $\Omega\subset\RR^n$, the {\it $k$-Hessian operator} $(k=1,...,n)$ is defined by the formula 
\[
S_k(D^2u)=\sigma_k(\Lambda)=\sum_{1\leq i_1<...<i_k\leq n}\lambda_{i_1}...\lambda_{i_k},
\] 
where $\Lambda=\Lambda(D^2u):=(\lambda_1,...,\lambda_n)$, the $\lambda$'s are the eigenvalues of $D^2u$ and $\sigma_k$ is the $k$-th elementary symmetric function. Equivalently, $S_k(D^2u)$ is the sum of the $k$-th principal minors of the Hessian matrix. See, {\it e.g.}, X.-J. Wang \cite{Wang94, Wang09}. These operators form an important class of second-order operators which contains, as the most relevant examples, the Laplace operator $S_1(D^2u)=\Delta u$  and the Monge-Amp\`{e}re operator $S_n(D^2u)=$ det $D^2u$. They are fully nonlinear when $k>1$. In particular, $S_2\left(D^2u\right)=\frac{1}{2}\left((\Delta u)^2-\abs{D^2 u}^2\right)$. The study of $k$-Hessian equations has many applications in geometry, optimization theory and in other related fields. See \cite{Wang09}. There exists a large literature about existence, regularity and qualitative properties of solutions for the $k$-Hessian equations, starting with the seminal work of L. Caffarelli, L. Nirenberg and J. Spruck \cite{CaNS85}. 

We point out that the $k$-Hessian operators are $k$-homogeneous and also invariant under rotations of coordinates. For more details about these operators, we refer to \cite{Wang09}.

As a by-product of our study, we construct self-similar solutions of a $k$-Hessian evolution equation posed on the whole Euclidean space. We establish that the self-similar solutions that we construct present similar properties with those self-similar solutions of the evolution equation $u_t=\Delta u^m$ in the slow diffusion range, $m>1$, called the porous medium equation and also with those solutions of the $p$-Laplacian evolution equation $u_t=\mbox{div}(\abs{\nabla u}^{p-2}\nabla u)$ when $p>2$.

The self-similar solutions are particular solutions that reflect some symmetries of the underlying equations (whenever they exist) and, although they are isolated objects related to simplified models, they play an important role both in the theory and in the applications. We refer the interested readers to the pioneering book of G.I. Barenblatt \cite{Barenblatt79} for a detailed discussion of this subject. We also point out that there exists an extensive literature about evolution equations that generalize the standard heat equation. This literature addresses, among others equations, the $p$-Laplacian equation, the porous medium equation and the space-fractional porous medium equation. See {\it e.g.} \cite{Bidaut-Veron09, FiWi08, FiWi16, GaVa04, Huang14, ISVa08, KaVa88, MaMe09, QuSo07, Vazquez07, Wang93}. 

Concerning exact solutions of some nonlinear diffusion equations, we note that in \cite{King90} new closed-form similarity solutions of $N$-dimensional radially symmetric equations were given, which are generalizations of the classical Barenblatt solutions. In \cite{ISVa08}, the authors studied an explicit equivalence between radially symmetric solutions for two basic models of nonlinear degenerate diffusion equations, namely, the porous medium equation and the $p$-Laplacian equation. The correspondence in \cite{ISVa08} between self-similar radial solutions is obtained by a careful and detailed phase plane analysis. In particular, Iagar et al. derive the existence of new self-similar solutions for the evolution $p$-Laplacian equation. In \cite{Bidaut-Veron09}, a complete classification of radial self-similar solutions of the $p$-Laplace heat equation is given in the case when $p>2$. The method of proof consists in performing a careful phase-plane analysis of the system associated to the second-order equation satisfied by the stationary part of the self-similar solution. In \cite{Huang14} several one-parameter families of explicit self-similar solutions were constructed for the porous medium equations with fractional operators, as for example, $u_t+(-\Delta)^s u^m=0$ and $u_t=\nabla\cdot(u^{m-1}\nabla(-\Delta)^{-s}u)$ for some ranges of $s$ and $m$. To derive the explicit self-similar solutions, certain special functions are used, such as the modified Bessel function and the (Gauss) hypergeometric function.

We prove the existence of radially symmetric solutions of 
\begin{equation}\label{eq:maineq}
S_k(D^2v)+\alpha v+\beta\xi\cdot\nabla v=0,\, v>0\;\; \mbox{in}\;\; \mathbb{R}^n,\; v(0)=a>0,
\end{equation}
where $\alpha,\beta$ are real parameters, $\cdot$ denotes the standard scalar product in $\mathbb{R}^n,\; \nabla v$ is the gradient of $v$ and $S_k(D^2v)$ is the $k$-Hessian operator of $v$. This equation is related to the study of the self-similar solutions of the $k$-Hessian evolution equation
\begin{equation}\label{eq:k-evol}
u_t=S_{k}(D^2 u)\;\; \mbox{in}\;\; (0,T)\times\mathbb{R}^n,
\end{equation}
which can be see as a nonlinear counterpart of the heat equation. Note that this equation satisfies a scaling group invariance: if $u$ is a solution, so is 
\[
(Su)(t,x)=\tilde{u}(t,x)=cu(at,bx),
\]
provided that $ac^{k-1}=b^{2k}$, where  $a,b$ and $c$ are arbitrary positive numbers (these are the scaling parameters). Those special solutions that are themselves invariant under the scaling group, i.e., $\tilde{u}=u$ are called {\it self-similar solutions}. If we impose an extra condition on the solutions, we obtain a corresponding condition on the parameters. Hence, the group of scaling is reduced to a one-parameter family. Note that we recover a well-known scaling for the heat equation ($k=1$ in \eqref{eq:k-evol}), namely $u_\lambda (t,x)=u(\lambda^2 t,\lambda x)$ for any positive number $\lambda$.

The configuration of the parameters $\alpha$ and $\beta$ in \eqref{eq:maineq} permit us to construct various types of self-similar solutions of equation \eqref{eq:k-evol}. Following the literature, these are classified into three types, according to the following: Let $v$ be the solution of \eqref{eq:maineq}, then the function
\begin{equation}\label{eq:Ansatz1}
u(t,x)=t^{-\alpha}v(xt^{-\beta})\;\;\;\; (\mbox{called of Type I})
\end{equation}
is a solution of \eqref{eq:k-evol} in $(0,\infty)\times\mathbb{R}^n$ if
\[
\alpha (k-1)+2k\beta=1.
\]
Now, for any $T>0$ the function
\begin{equation}\label{eq:Ansatz2}
u(t,x)=(T-t)^\alpha v(x(T-t)^\beta)\;\;\;\; (\mbox{called of Type II})
\end{equation}
is a solution of \eqref{eq:k-evol} in $(0,T)\times\mathbb{R}^n$ if
\[
\alpha (k-1)+2k\beta=-1.
\]
Finally, the function 
\begin{equation}\label{eq:Ansatz3}
u(t,x)=e^{-\alpha t}v(xe^{-\beta t})\;\;\;\; (\mbox{called of Type III})
\end{equation}
is an eternal (defined for all times) solution of \eqref{eq:k-evol} in $(-\infty,\infty)\times\mathbb{R}^n$ if
\[
\alpha (k-1)+2k\beta=0.
\]
As we noted above, the self-similar solutions usually describe the asymptotic behavior of the general solutions of a diffusion equation. Thus, they are of great importance in the general theory of similar equations. We hope that the study of the solutions of \eqref{eq:maineq} is a first step for understand the behavior of solutions of \eqref{eq:k-evol}. 

On the other hand, if the parameters $\alpha$ and $\beta$ are related via 
\begin{equation}\label{eq:sse}
\alpha (k-1)+2k\beta=\rho
\end{equation}
where $\rho\in\{1,-1,0\}$ then they are called {\it self-similar exponents}. The corresponding solutions $v$ of \eqref{eq:maineq} are then called {\it self-similar profiles}, or simply {\it profiles}. The relation \eqref{eq:sse} between the two exponents introduced by \eqref{eq:Ansatz1}-\eqref{eq:Ansatz3} arises from the requirement that $v$ satisfies an equation involving only $\xi$, i.e., equation \eqref{eq:maineq}. Note that, independently of the form of the solutions $u(t,x)$ in ansatz \eqref{eq:Ansatz1}-\eqref{eq:Ansatz3}, the profile $v$ satisfies the same {\it profile equation} \eqref{eq:maineq}.

In situations where the problem under consideration satisfies a conservation law (e.g. conservation of mass), the law in question supplies a second condition on $\alpha$ and $\beta$, thus fixing the parameters. In this case we speak of {\it self-similar solutions of the first kind}. When there are no further restrictions on $\alpha$ and $\beta$, one of the exponents $\alpha$ or $\beta$ is free. In such a case, we speak of {\it self-similar solutions of the second kind}. See \cite{Barenblatt79}.

For a non-negative function $u(t,x)$, we define the function $M(t)$ by
\begin{equation}\label{eq:mass}
M(t)=\int_{\mathbb{R}^n}u(t,x)dx.
\end{equation}
In some physical models, $M(t)$ represents (when $u(t,\cdot)$ is integrable on $\mathbb{R}^n$) the total {\it mass} in $\mathbb{R}^n$ at time $t$ of the solution $u(t,x)$ of an evolution PDE. From conservation principles, as we noted above, this quantity remains invariant, that is, it does not depend on the time variable. In Section \ref{Exact aeqnb}, we construct explicit solutions of type I and type II along which $M(t)$ is constant.

We say that a classical solution of \eqref{eq:k-evol} {\it blows up at time} $T<\infty$ if 
\[
\lim_{t\uparrow T}\sup\abs{u(t,\cdot)}=\infty.
\]
We give examples of solutions of type II, which exhibit finite time blow up for any $x$ in a compact domain or on the whole space.
In particular, for the classical heat equation, an infinite family of self-similar solutions of type II that blows up in finite time is given in Section \ref{buHE}.

Now, for $r=\abs{\xi}$, the radially symmetric profile $v$ satisfies the ODE:
\begin{equation}\label{eq:goveq}
c_{n,k}r^{1-n}(r^{n-k}(v')^k)'+\alpha v+\beta rv'=0,\, v>0\;\; \mbox{on}\;\; (0,\infty),
\end{equation}
where prime denotes differentiation with respect to $r>0$ and 
\begin{equation}
\label{eq:inicond}
\begin{cases}
v(0)=a,\\
v'(0)=0.
\end{cases}
\end{equation}

The results obtained here depends strongly on the relation between the parameters $\alpha, \beta, k$ and the dimension $n$, as well as on whether $k$ is an odd or an even integer.

\begin{theorem}\label{mainth}
Let $a>0$ and let $\alpha,\beta\in\mathbb{R}$. Let $k$ be an odd integer, $1\leq k\leq n$. Assume that $\alpha$ and $\beta$ satisfy the conditions
\begin{equation}\label{eq:maincond}
\alpha\leq\frac{\beta(n-2k)}{k}\;\;\mbox{and}\;\;\beta>0. 
\end{equation}
Then there exists a unique solution $v$ of \eqref{eq:goveq} and \eqref{eq:inicond} on $(0,\infty)$. Moreover, for $\alpha\neq 0,\delta=\beta/\alpha$,
 the function
\[
E(r)=r^2v(r)^{2\delta}
\]
is such that $E'(r)>0$ for all $r>0$.
\end{theorem}
We will prove this result adapting to our setting a technique of S.-Y. Hsu (Lemma 2.1 in \cite{Hsu12}), which avoids the use of the phase-plane method. 

\begin{remark}
Although the above result only includes the odd-Hessian, we also find explicit solutions for the even-Hessian for certain values of the parameters $\alpha$ and $\beta$. This includes the important case when $\alpha=n\beta$, for which we unify these $k$-Hessian operators. See Section \ref{Exact aeqnb}.
\end{remark}

The plan of the paper is as follows. In Section \ref{Existence} we will establish the existence and prove various properties of the solutions of \eqref{eq:goveq}, \eqref{eq:inicond}. The proof of Theorem \ref{mainth} appears in this section. In Section 2 we also give explicit self-similar solutions for particular values of the parameters. Section \ref{Exact aeqnb} is devoted to constructing explicit solutions in self-similar form for equation \eqref{eq:k-evol} when $\alpha=n\beta$. Using the results of Section 3, in Section \ref{familybu} we obtain an explicit family of self-similar solutions of type II that blow up in finite time. In Section \ref{buHE} we find, for the heat equation, a whole family of self-similar solutions that blow up in finite time. This family is given in terms of a confluent hypergeometric function of the first kind, which is called the Kummer function. Finally, the local existence of a solution is proved in the Appendix.

\section{Existence of solutions}\label{Existence}
In this section we prove the existence of radially symmetric solutions $v$ of equation \eqref{eq:maineq} which satisfy $v(0)=a$.
Before giving the proof of the existence result, we give an example of explicit solutions for particular values of the parameters $\alpha$ and $\beta$. Note that if $\alpha=\beta=0$, then the constant function $v=a$ is the unique solution of \eqref{eq:goveq} and \eqref{eq:inicond}, and $u$ is a constant solution of \eqref{eq:k-evol}. 

Let $\alpha=0$ and $\beta\neq 0$. Here equation \eqref{eq:goveq} can be solved explicitly: either $v'\equiv 0$ (hence $v=a$ and $u$ is a constant solution of
\eqref{eq:k-evol}, as in case $\alpha=\beta=0$), or there exist explicit profiles
\begin{equation}
v(r)=
\label{eq:explicitv}
\begin{cases}
a-C_+r^{\frac{2k}{k-1}}\;\;\mbox{in}\;\; (0,\overline{r})\;\;\mbox{if}\;\; \beta>0\; \mbox{and } k\; \mbox{even},\\
a+C_-r^{\frac{2k}{k-1}}\;\;\mbox{in}\;\; (0,\infty)\;\;\mbox{if}\;\; \beta<0,
\end{cases}
\end{equation}
where $C_{\pm}=\left(\frac{k-1}{2k}\right)\left(\frac{(k-1)(\pm\beta)}{(n(k-1)+2k)c_{n,k}}\right)^{\frac{1}{k-1}}$ and $\overline{r}=\left(\frac{a}{C_+}\right)^{\frac{k-1}{2k}}$.
In particular, for these profiles we have two self-similar solutions of \eqref{eq:k-evol}, one of type I and other of type II. The corresponding exponents $\beta$ are $\frac{1}{k}$ and $-\frac{1}{k}$, respectively. More precisely,
\begin{equation}\label{eq:betaposi}
u(t,x)=
\begin{cases}
a-C_+t^{-\frac{1}{k-1}}\abs{x}^{\frac{2k}{k-1}},\;\; 0\leq \abs{x}<t^{\frac{1}{2k}}\overline{r},\; k\; \mbox{even},\\
0,\;\;\abs{x}\geq t^{\frac{1}{2k}}\overline{r}
\end{cases}
\end{equation}
and
\begin{equation}\label{eq:betanega}
u(t,x)=a+C_-(T-t)^{-\frac{1}{k-1}}\abs{x}^\frac{2k}{k-1},\;\; (t,x)\in (0,T)\times\mathbb{R}^n.
\end{equation}

\begin{remark}
Note that the solution in \eqref{eq:betaposi} has compact support in space and this support expands, that is, increases from $\{0\}$ as $t$ increases from $0$. In fact, supp\,$u(t,\cdot)\subseteq B\left(0,\,t^{\frac{1}{2k}}\overline{r}\right)$. On the other hand, the solution in \eqref{eq:betanega} is defined on the whole space $\mathbb{R}^n$ and blows up everywhere in $x$ as $t\uparrow T$.
\end{remark}

\begin{lemma}\label{vbeha}
Let $k$ be an odd integer, $1\leq k\leq n$. Let $\alpha,\beta\neq 0$. Assume that $\alpha$ and $\beta$ satisfy inequality
\begin{equation}\label{eq: anybeta}
\frac{k\alpha}{\beta}\leq n-2k.
\end{equation}
For any $R_0>0$ and $a>0$, let $v$ be the solution of \eqref{eq:goveq} and \eqref{eq:inicond} in $(0,R_0)$. Then
\begin{equation}\label{eq:vineq}
v(r)+\delta rv'(r)>0\;\;\mbox{in}\;\; [0,R_0)
\end{equation}
and
\begin{equation}
\label{eq:singvprime}
\begin{cases}
v'(r)<0\;\;\mbox{in}\;\; (0,R_0)\;\;\mbox{if}\;\; \alpha>0,\\
v'(r)>0\;\;\mbox{in}\;\; (0,R_0)\;\;\mbox{if}\;\; \alpha<0.
\end{cases}
\end{equation}
\end{lemma}
\begin{proof}
Let $h(r)=v(r)+\delta rv'(r)$. By \eqref{eq:maincond}, we have $n-2k\geq k/\delta$. Then, by direct computation,
\begin{equation}\label{eq:hineq}
h'+\left(\frac{\frac{n-2k}{k}-\frac{1}{\delta}}{r}+\frac{\beta}{kc_{n,k}}r^k(v')^{1-k}\right)h=\frac{\frac{n-2k}{k}-\frac{1}{\delta}}{r}v\geq 0\;\;\mbox{in}\;\; (0,R_0).
\end{equation}
Now let 
\begin{equation}\label{eq:efe}
f(r)=\exp\left(\frac{\beta}{kc_{n,k}}\int_0^r\rho^k(v'(\rho))^{1-k}d\rho\right).
\end{equation}
Note that the function $f(r)$ is well-defined since the integrand has a finite limit at $\rho=0$. In fact, using \eqref{eq:goveq}, \eqref{eq:inicond} and L'Hopital's rule, we have 
\[
\lim_{r\downarrow 0}\left(\frac{v'(r)}{r}\right)^k=\lim_{r\searrow 0}\frac{-\alpha}{c_{n,k}}\frac{r^{n-1}h(r)}{nr^{n-1}}=\frac{-\alpha a}{nc_{n,k}}\in\mathbb{R}\setminus\{0\}.
\]
By \eqref{eq:hineq}
\begin{eqnarray*}
(r^{\frac{n-2k}{k}-\frac{1}{\delta}}f(r)h(r))'&\geq 0&\;\forall\; 0<r<R_0\\
&\Rightarrow& r^{\frac{n-2k}{k}-\frac{1}{\delta}}f(r)h(r)>0\;\forall\; 0<r<R_0\\
&\Rightarrow& h(r)>0\;\forall\; 0<r<R_0
\end{eqnarray*}
and \eqref{eq:vineq} follows. By \eqref{eq:goveq}, \eqref{eq:inicond} and \eqref{eq:vineq},
\begin{eqnarray*}
c_{n,k}r^{1-n}(r^{n-k}(v')^k)'&=&-\alpha h(r)\begin{cases}
<0\;\;\mbox{in}\;\; (0,R_0)\;\;\mbox{if}\;\; \alpha>0,\\
>0\;\;\mbox{in}\;\; (0,R_0)\;\;\mbox{if}\;\; \alpha<0
\end{cases}\\
&\Rightarrow&\begin{cases}
r^{n-k}(v')^k<0\;\;\mbox{in}\;\; (0,R_0)\;\;\mbox{if}\;\; \alpha>0,\\
r^{n-k}(v')^k>0\;\;\mbox{in}\;\; (0,R_0)\;\;\mbox{if}\;\; \alpha<0
\end{cases}
\end{eqnarray*}
and \eqref{eq:singvprime} follows, since $k$ is an odd integer.
\end{proof}

We are now ready to prove Theorem \ref{mainth}.

\begin{proof}[\bf Proof of Theorem \ref{mainth}.] 
We may assume that $\alpha\neq 0$, since the solutions are given explicitly by \eqref{eq:betaposi}-\eqref{eq:betanega} when $\alpha=0$.
We note that the uniqueness of the solutions of \eqref{eq:goveq} and \eqref{eq:inicond} in $(0,\infty)$ follows from the standard ODE theory. Hence we only need to prove the existence of solutions of \eqref{eq:goveq} and \eqref{eq:inicond} in $(0,\infty)$. The local existence of solutions of \eqref{eq:goveq} and \eqref{eq:inicond} in a neighbourhood of the origin follows from classical arguments involving the Banach fixed point theorem (see the Apenndix).

Let $(0,R_0)$ be the maximal interval of the existence of solution of \eqref{eq:goveq} and \eqref{eq:inicond}. Suppose $R_0<\infty$. Then there exists a sequence $\{r_i\}_{i=1}^{\infty},\, r_i\uparrow R_0$ as $i\rightarrow\infty$ such that either
\[
\abs{v'(r_i)}\rightarrow\infty\;\; \mbox{as}\;\; i\rightarrow\infty\;\; \mbox{or}\;\; v(r_i)\downarrow 0\;\; \mbox{as}\;\; i\rightarrow\infty\;\; \mbox{or}\;\; v(r_i)\rightarrow\infty\;\; \mbox{as}\;\; i\rightarrow\infty.
\]
By Lemma \ref{vbeha}, we have
\begin{equation}\label{eq:Eprima}
E'(r)=2rv^{2\delta}+2\delta r^2v^{2\delta-1}v'=2rv^{2\delta-1}(v+\delta rv')>0\;\; \forall\;\; 0<r<R_0.
\end{equation}
We divide the proof into two cases, depending on the sign of $\alpha$.

Case 1: $\alpha>0$. 

By \eqref{eq:maincond}, inequality \eqref{eq: anybeta} holds. Hence by \eqref{eq:Eprima},
\[
E(r)=r^2v^{2\delta}\geq E(R_0/2)>0\;\;\forall\;\; R_0/2\leq r<R_0
\]
\begin{equation}\label{eq:vlowerbound}
\Rightarrow v(r)\geq (R_0^{-2}E(R_0/2))^{\frac{1}{2\delta}}\;\;\forall\;\; R_0/2\leq r<R_0
\end{equation}
By Lemma \ref{vbeha}, we have $v'<0$ on $(0,R_0)$.  Hence
\begin{equation}\label{eq:vupperbound}
0<v(r)\leq v(0)=a\;\;\forall\;\; 0\leq r<R_0.
\end{equation}
By \eqref{eq:goveq}, \eqref{eq:inicond} and \eqref{eq:vupperbound}, if $0<r<R_0$, we have
\begin{eqnarray*}
&&c_{n,k}r^{1-n}(r^{n-k}(v')^k)'=-(\alpha v+\beta rv')\\
&\Rightarrow& c_{n,k}r^{n-k}(v')^k=-\left(\alpha\int_0^r\rho^{n-1}v(\rho)d\rho+\beta\int_0^r\rho^{n}v'(\rho)d\rho\right)
\end{eqnarray*}
\begin{equation}\label{eq:byparts}
\Rightarrow c_{n,k}r^{n-k}(v')^k=-\beta r^nv(r)+(n\beta-\alpha)\int_0^r\rho^{n-1}v(\rho)d\rho
\end{equation}
\begin{eqnarray*}
&\Rightarrow& c_{n,k}(v')^k=-\beta r^kv(r)+\frac{n\beta-\alpha}{r^{n-k}}\int_0^r\rho^{n-1}v(\rho)d\rho\\
&\Rightarrow& c_{n,k}\abs{v'(r)}^k\leq\left(\beta+\frac{\abs{n\beta-\alpha}}{n}\right)R_0^kv(0)
\end{eqnarray*}
\begin{equation}\label{eq:vprimeupperbound}
c_{n,k}\abs{v'(r)}\leq\left(\beta+\frac{\abs{n\beta-\alpha}}{n}\right)^\frac{1}{k}R_0v(0)^\frac{1}{k}.
\end{equation}
By \eqref{eq:vlowerbound}, \eqref{eq:vupperbound} and \eqref{eq:vprimeupperbound}, we have obtained a contradiction. Hence no such sequence $\{r_i\}_{i=1}^{\infty}$ exists. Thus $R_0=\infty$ and there exists a unique solution of \eqref{eq:goveq} and \eqref{eq:inicond} in $(0,\infty)$.

Case 2: $\alpha<0$. 
By Lemma \ref{vbeha},
\begin{equation}\label{eq:vvprime}
0<v'(r)\leq\frac{v(r)}{\abs{\delta}r}\;\; \mbox{in}\;\; (0,R_0).
\end{equation}
Choose $r_0\in(0,R_0)$ and let $C_0=\max_{0\leq r\leq r_0}v'(r)$. Then by \eqref{eq:vvprime}
\[
0<v'(r)\leq C_0+\frac{v(r)}{\abs{\delta}r}\leq Cv(r)\;\; \forall\;\; r\in(0,R_0),
\]
where $C=\frac{C_0}{v(0)}+(\abs{\delta}r_0)^{-1}>0$. Then
\begin{equation}\label{eq:vbound}
v(0)\leq v(r)\leq v(0)\exp(CR_0)\;\; \forall\;\; 0\leq r<R_0
\end{equation}
and
\begin{equation}\label{eq:vprimebound}
0<v'(r)\leq Cv(0)\exp(CR_0)\;\; \forall\;\; 0\leq r<R_0.
\end{equation}
By \eqref{eq:vbound} and \eqref{eq:vprimebound}, we again obtain a contradiction. Hence no such sequence $\{r_i\}_{i=1}^{\infty}$ exists. Thus $R_0=\infty$ and there exists a unique solution of \eqref{eq:goveq} and \eqref{eq:inicond} in $(0,\infty)$. By \eqref{eq:Eprima} and Cases 1 and 2, the theorem follows.
\end{proof}

\begin{lemma}\label{lem:betacero}
Let $k$ odd. Let $a>0, \alpha<0$ and $\beta=0$. Then \eqref{eq:goveq} and \eqref{eq:inicond} have a solution in $(0,\infty)$.
\end{lemma}
\begin{proof}
Let $a>0$ and denote by $(0,R_0)$ the maximal interval where a solution $v$ of \eqref{eq:goveq} and \eqref{eq:inicond} exists. Suppose $R_0<\infty$. Then there exists a sequence $\{r_i\}_{i=1}^{\infty},\, r_i\uparrow R_0$ as $i\rightarrow\infty$ such that either
\[
\abs{v'(r_i)}\rightarrow\infty\;\; \mbox{as}\;\; i\rightarrow\infty\;\; \mbox{or}\;\; v(r_i)\downarrow 0\;\; \mbox{as}\;\; i\rightarrow\infty\;\; \mbox{or}\;\; v(r_i)\rightarrow\infty\;\; \mbox{as}\;\; i\rightarrow\infty.
\]
Multiplying \eqref{eq:goveq} by $r^{n-1}$ and taking $\beta=0$, we have
\begin{equation}\label{eq:Fprima}
(c_{n,k}r^{n-k}(v'(r))^k)'=\abs{\alpha}r^{n-1}v(r)>0\;\; \forall\;\; 0<r<R_0.
\end{equation}
Since $k$ is odd, it follows that $v'>0$ on $(0,R_0)$. Hence
\begin{equation}\label{eq:vlowerbound2}
0<a=v(0)\leq v(r)\;\;\forall\;\; 0\leq r<R_0.
\end{equation}

Integrating \eqref{eq:Fprima}, if $0<r<R_0$, 
\begin{eqnarray*}
&&r^{n-k}(v'(r))^k=\frac{\abs{\alpha}}{c_{n,k}}\int_0^r s^{n-1}v(s)ds\\
&\Rightarrow&r^{n-k}(v'(r))^k\leq\frac{\abs{\alpha}}{nc_{n,k}}r^n v(r)\\
&\Rightarrow&(v'(r))^k\leq\frac{\abs{\alpha}}{nc_{n,k}}r^k v(r)\\
\end{eqnarray*}
\begin{equation*}
\Rightarrow 0<v'(r)\leq\left(\frac{\abs{\alpha}}{nc_{n,k}}\right)^{\frac{1}{k}}r (v(r))^\frac{1}{k}.
\end{equation*}
Then 
\begin{equation}\label{eq:vbound2}
v(0)\leq v(r)\leq\left(v(0)^\frac{k-1}{k}+\left(\frac{\abs{\alpha}}{nc_{n,k}}\right)^{\frac{1}{k}}\frac{R_{0}^2}{2}\right)^\frac{k}{k-1}:=C\;\;\forall\;\; 0\leq r<R_0
\end{equation}
and 
\begin{equation}\label{eq:vprimabound2}
0<v'(r)\leq\left(\frac{\abs{\alpha}C}{nc_{n,k}}\right)^{\frac{1}{k}}R_0\;\;\forall\;\; 0\leq r<R_0.
\end{equation}
By \eqref{eq:vlowerbound2}, \eqref{eq:vbound2} and \eqref{eq:vprimabound2}, we obtain a contradiction. Hence no such sequence $\{r_i\}_{i=1}^{\infty}$ exists. Thus $R_0=\infty$ and there exists a unique solution of \eqref{eq:goveq} and \eqref{eq:inicond} in $(0,\infty)$.
\end{proof}
In particular, when $k>1$ is an odd integer, we obtain a self-similar solution of type II for equation \eqref{eq:k-evol}, which blows up in finite time. The solution is given in separated variables of the form
\[
u(t,x)=(T-t)^{-\frac{1}{k-1}}v(x),
\]
where $v(x)=v(\abs{x})$ is the unique positive solution of the equation $S_k(D^2 v)=\frac{1}{k-1}v$ in $\mathbb{R}^n$. 

\begin{lemma}\label{lem:negativeexpo}
Let $k$ odd. Let $a>0$ and $0>n\beta\geq\alpha$. Then \eqref{eq:goveq} and \eqref{eq:inicond} have a solution in $(0,\infty)$.
\end{lemma}
\begin{proof}
Let $a>0$ and denote by $(0,R_0)$ the maximal interval where a solution $v$ of \eqref{eq:goveq} and \eqref{eq:inicond} exists. Suppose $R_0<\infty$. Then there exists a sequence $\{r_i\}_{i=1}^{\infty},\, r_i\uparrow R_0$ as $i\rightarrow\infty$ such that either
\[
\abs{v'(r_i)}\rightarrow\infty\;\; \mbox{as}\;\; i\rightarrow\infty\;\; \mbox{or}\;\; v(r_i)\downarrow 0\;\; \mbox{as}\;\; i\rightarrow\infty\;\; \mbox{or}\;\; v(r_i)\rightarrow\infty\;\; \mbox{as}\;\; i\rightarrow\infty.
\]
We only consider the case $0>n\beta>\alpha$ (since the case $0>n\beta=\alpha$ is similar). By \eqref{eq:byparts}, $c_{n,k}r^{n-k}(v'(r))^k>0$ for all $0<r<R_0$. Hence $v'(r)>0$ for all $0<r<R_0$. Then by \eqref{eq:byparts},
\[
c_{n,k}r^{n-k}(v')^k\leq-\beta r^nv(r)-(\alpha-n\beta)\int_0^r\rho^{n-1}v(r)d\rho=\frac{\abs{\alpha}}{n}r^nv(r)	
\]
The rest of the proof is as in Lemma \ref{lem:betacero}, with the obvious changes in case $k=1$.
\end{proof}

\section{Exact solutions for $\alpha=n\beta\neq 0$}\label{Exact aeqnb}
Let $k\neq 1$. In this section we consider the case $\alpha=n\beta$. We construct explicit solutions in self-similar form. By \eqref{eq:byparts}, we obtain the separable ODE
\begin{equation}\label{eq:exacteq}
c_{n,k}r^{n-k}(v'(r))^k=-\beta r^nv(r).
\end{equation}
We solve this equation under the conditions $v>0$ and $v(0)=a>0$. For this, we consider two cases.
\medskip

{\bf Case 1}: $\alpha<0$. From \eqref{eq:exacteq}, we obtain
\[
(v'(r))^k=\frac{\abs{\beta}}{c_{n,k}}r^kv(r)=\frac{\abs{\alpha}}{nc_{n,k}}r^kv(r).
\]
Now, according to the parity of $k$, we have two subcases.
\begin{itemize}
\item[$\bullet$] If $k$ is odd, then
\[
v'(r)=\left(\frac{\abs{\alpha}}{nc_{n,k}}\right)^\frac{1}{k}rv(r)^\frac{1}{k}>0.
\]
Integrating this equation from 0 to $r$, we obtain
\begin{eqnarray*}
\frac{k}{k-1}\left(v^\frac{k-1}{k}-a^\frac{k-1}{k}\right)&=&\left(\frac{\abs{\alpha}}{nc_{n,k}}\right)^\frac{1}{k}\frac{r^2}{2}\\
v^\frac{k-1}{k}&=&a^\frac{k-1}{k}+\frac{k-1}{2k}\left(\frac{\abs{\alpha}}{nc_{n,k}}\right)^\frac{1}{k}r^2\\
v(r)&=&\left(a^\frac{k-1}{k}+\frac{k-1}{2k}\left(\frac{\abs{\alpha}}{nc_{n,k}}\right)^\frac{1}{k}r^2\right)^\frac{k}{k-1}\;\; \forall r>0.
\end{eqnarray*}

\item[$\bullet$] If $k$ is even, then
\[
\abs{v'(r)}=\left(\frac{\abs{\alpha}}{nc_{n,k}}\right)^\frac{1}{k}rv(r)^\frac{1}{k}.
\]
Note that if $v'(r)>0$, then we arrive at the same solution as in the case $k$ odd. 

If $v'(r)<0$, we have
\[
v^{-\frac{1}{k}}v'=-\left(\frac{\abs{\alpha}}{nc_{n,k}}\right)^\frac{1}{k}r<0.
\]
Integrating this equation from 0 to $r$ (as long as $v>0$), we obtain
\[
v(r)=\left(a^\frac{k-1}{k}-\frac{k-1}{2k}\left(\frac{\abs{\alpha}}{nc_{n,k}}\right)^\frac{1}{k}r^2\right)^\frac{k}{k-1}\;\; \forall\;\;  0\leq r<r^*,
\]
where $r^*$ is the positive root of the quadratic equation $a^\frac{k-1}{k}-\frac{k-1}{2k}\left(\frac{\abs{\alpha}}{nc_{n,k}}\right)^\frac{1}{k}r^2=0$.
Recall that we are looking for a solution defined on the whole space. So, if we do not impose the condition $v>0$, then the solution may take zero or negative values. To avoid this, we can cut off the unwanted part of the solution. To do so, define $v(r)=0$ for all $r\geq r^*$. This has the effect of extending the solution $v$ to all $r\geq 0$.
\end{itemize}

{\bf Case 2}: $\alpha>0$. Clearly $k$ odd is allowed and thus $v'(r)<0$ (as long as $v>0$). In this case, from \eqref{eq:exacteq}, we obtain
\[
(-v'(r))^k=\frac{\alpha}{nc_{n,k}}r^kv(r)>0.
\]
Solving this equation we obtain the same solution as in Case 1, where $k$ is even and $\alpha<0$.

We conclude that, when $v'<0$ and $\alpha=n\beta$, it is possible to unify Cases 1 and 2 for all $k$. Further, in this case we can write 
\begin{equation}\label{eq:vprofile}
v(r)=\left(C-\frac{k-1}{k}\left(\frac{\beta}{c_{n,k}}\right)^\frac{1}{k}\frac{r^2}{2}\right)_+^{\frac{k}{k-1}},\;\; r\geq 0,
\end{equation}
where $C=a^\frac{k-1}{k}$ and $(\cdot)_+$ denotes the positive part. Observe that this class of profiles is similar to the classical class of Barenblatt profiles for the porous medium equation ($m>1$) and the $p$-Laplacian evolution equation ($p>2$), as we noted in the Introduction. See, e.g., \cite{ISVa08}. 

Note that the quantity $M(t)$ defined in \eqref{eq:mass} is conserved for self-similar solutions when $\alpha=n\beta$ and the profile $v$ is integrable in $\mathbb{R}^n$, such as in \eqref{eq:vprofile}. This is the case for self-similar solutions of type I with $\beta=\frac{1}{n(k-1)+2k}>0$, and also type II with $\beta=-\frac{1}{n(k-1)+2k}<0$. However, $M(t)$ is not conserved for eternal solutions of type III for any $\beta\neq 0$ since, in the case of our equation, for those solutions the exponents $\alpha$ and $\beta$ have opposite sings.

In case of type I solutions, we recover the family of self-similar positive solutions with finite \lq\lq mass'' found in \cite{Sanchez20}. For completeness, we recall its explicit form:
\begin{equation}\label{eq:k-Baren}
U_C(t,x)=t^{-\alpha}\left(C-\gamma\left(\frac{\abs{x}}{t^\beta}\right)^2\right)_{+}^\frac{k}{k-1},
\end{equation}
where $(\cdot)_+$ denotes the positive part, $C>0$ is an arbitrary constant, and $\alpha, \beta$ and $\gamma$ have explicit values, namely
\[
\alpha=\frac{n}{n(k-1)+2k},\;\;\;\; \beta=\frac{1}{n(k-1)+2k},\;\;\;\; \gamma=\frac{k-1}{2k}\left(\frac{\beta}{c_{n,k}}\right)^\frac{1}{k},\;\;\;\; c_{n,k}=\frac{1}{n}\binom{n}{k}.
\] 
Note that this family is well defined for the full range of $k$-Hessian operators with $k\neq 1$. In the language of \cite{Bidaut-Veron09}, $U_C$ in \eqref{eq:k-Baren} has an expanding support. 

Finally, we can interpret the initial condition $U_C(0,x)$ as $M\delta_0(x)$ in the sense that $U_C(t,x)\rightarrow M\delta_0(x)$ as $t\downarrow 0$, where $\delta_0(x)$ is the Dirac delta function concentrated at 0 and $M=\int_{\RR^n} U_C(t,x)\,dx$. In other words,
\[
\lim_{t\downarrow 0}\int_{\RR^n}U_C(t,x)\varphi(x)dx=M\varphi(0)\,\, \mbox{for all } \varphi\in C_{c}^\infty(\RR^n). 
\]
This statement is a particular case of Lemma 5.4 in \cite{Vazquez03} applied to our setting.

\section{A family of blow up solutions}\label{familybu}
We now give some explicit type II self-similar blow up solutions for equation \eqref{eq:k-evol}. By \eqref{eq:Ansatz2}, they have the form
\begin{equation}\label{blowup}
u(t,x)=(T-t)^\alpha v(\xi),\;\;\xi=x(T-t)^\beta,\, t\in [0,T),\,T<\infty,\; x\in \RR^n.
\end{equation}
If there is only one condition on the scaling exponents $\alpha$ and $\beta$, which here is given by \eqref{eq:sse} with $\rho=-1$, that is, $\alpha (k-1)+2k\beta=-1$, then it is not enough to determine $\alpha$ and $\beta$ explicitly. In particular, when $k=1$, we obtain $\beta=-\frac{1}{2}$ leaving $\alpha$ free. Thus, for a radially symmetric function $v$, \eqref{eq:goveq} reduces to the equation
$
-\alpha v+\frac{1}{2}\,r v'=r^{1-n}(r^{n-1}v')'.
$
However, the choice $\alpha=-\frac{n}{2}(=n\beta)$ allows for the integration of the last equation easily, yielding the profile $v(r)=a\,e^{\frac{r^2}{4}}$, where $a>0$ is a  constant. Inserting all the information in \eqref{blowup}, we write a self-similar solution for equation \eqref{eq:k-evol} that blows up in finite time $T$ and has infinite mass
\[
u(t,x)=a(T-t)^{-\frac{n}{2}}e^{\frac{\abs{x}^2}{4(T-t)}}.
\]
This solution has the simplest form. In the next section, we obtain a whole family of blows up solutions for the heat equation. 

Now let $k\neq 1$ and $\alpha=n\beta$. Using the results obtained in the previous section, we can write in a closed form the profiles and then derive explicit self-similar solutions which blow up in finite time. A surprising fact is that when $k$ is an even integer we have a profile with compact support and when  $k$ is an odd integer we have a profile without compact support. Since the profiles are radial, necessarily the support is a ball or the whole space. Summing up, the following are self-similar solutions of type II for equation \eqref{eq:k-evol} that blow up in finite time $T$:
\begin{eqnarray*}
u(t,x)&=&(T-t)^{\alpha}\left(a^{\frac{k}{k-1}}-\left(\frac{k-1}{2k}\right)\left(\frac{\abs{\beta}}{c_{n,k}}\right)^\frac{1}{k}\frac{\abs{x}^2}{(T-t)^{2\abs{\beta}}}\right)_{+}^{\frac{k}{k-1}}\;\; (k\; \mbox{even}),\;\;\\
u(t,x)&=&(T-t)^{\alpha}\left(a^{\frac{k}{k-1}}+\left(\frac{k-1}{2k}\right)\left(\frac{\abs{\beta}}{c_{n,k}}\right)^\frac{1}{k}\frac{\abs{x}^2}{(T-t)^{2\abs{\beta}}}\right)^{\frac{k}{k-1}}\;\; (k\; \mbox{odd}),
\end{eqnarray*}
where  $a>0$ is an arbitrary constant, $\alpha=-\frac{n}{n(k-1)+2k}$ and $\beta=-\frac{1}{n(k-1)+2k}$.
 
\section{A family of blow up solutions for the heat equation}\label{buHE}
In this section we find, for the heat equation, a whole family of self-similar solutions of type II that blows up in finite time. Recall that, from the previous section, the choice $k=1$ and $\rho=-1$ in \eqref{eq:sse} fixes the value of the exponent $\beta$ to $\beta=-\frac{1}{2}$ leaving $\alpha$ free. We are thus left with a continuum of admissible scaling exponents $\alpha<0$, as is typical for self-similarity of the second kind. A discretely infinite sequence of exponents $\alpha_m$ is, however, selected by requiring that the family be simpler to compute.

The ODE equation for the profile $v$ is $-\alpha v+\frac{1}{2}\,r v'=r^{1-n}(r^{n-1}v')'$. In order to obtain more precise information about the solutions, we rewrite this equation as 
\begin{equation}\label{eq:soe}
r^2v''+r\left(n-1-\frac{r^2}{2}\right)v'+\alpha r^2 v=0.
\end{equation}
Since our attention is focused on blow up solutions, we set $\alpha=-\gamma<0$, with the restriction $\gamma\geq\frac{n}{2}$ due to Lemma \ref{lem:negativeexpo}. To find explicit solutions, we introduce the new variables 
\[
z=\frac{r^2}{4},\; w(z)=v(r),
\]
with which equation \eqref{eq:soe} takes the form 
\[
z\frac{d^2w}{dz^2}+\left(\frac{n}{2}-z\right)\frac{dw}{dz}-\gamma w=0.
\]
The above is a particular case of the so-called {\it confluent hypergeometric equation}
\begin{equation}\label{eq:che}
z\frac{d^2w}{dz^2}+\left(b-z\right)\frac{dw}{dz}-a w=0,
\end{equation}
with $a$ and $b$ constants. In general, the parameters $a, b$ and the variables $z, w$ may take complex values. This is also known as {\it Kummer’s equation}, which is one of the most important differential equations in physics, chemistry, and engineering. Equation \eqref{eq:che} has a regular singularity at $z=0$ with index 0. Thus, we have a power series representing $w$. This power series is called the {\it confluent hypergeometric function of the first kind}, called {\it Kummer's function}, which is defined by 
\begin{equation}\label{eq:Kummersolution}
M(a,b;z)=\sum_{s=0}^{\infty}\frac{(a)_s}{(b)_s} \frac{z^s}{s!},
\end{equation}
where the {\it Pochhammer symbol} $(a)_s$ is defined by
\[
(a)_0=1,\,(a)_1=a,\,\mbox{and }\, (a)_s=a(a+1)\cdot\cdot\cdot(a+s-1),\, \mbox{for}\, s\in\mathbb{N}.
\]
Note that $M(a,b;z)$ is an entire function of $z$ (provided it exists). Among the many properties of this function, we highlight that $M(a,b;0)=1$ (if $b$ is not a non-positive integer) and $M(a,a;z)=e^z$. 

When $a=\gamma (=-\alpha),\, b=\frac{n}{2}$ and $z=\frac{r^2}{4}$ in \eqref{eq:Kummersolution}, an explicit solution of equation \eqref{eq:soe} is
\begin{equation}\label{eq:profilesolution}
v(r)=M\left(-\alpha,\frac{n}{2};\frac{r^2}{4}\right)=\sum_{s=0}^{\infty}\frac{(-\alpha)_s}{(\frac{n}{2})_s} \frac{\left(\frac{r^2}{4}\right)^s}{s!}.
\end{equation}
By the asymptotic behavior of the Kummer function as $\abs{z}\rightarrow\infty$ (if $\mbox{Re}(z)>0$), these solutions have very large growth as $r\rightarrow\infty$. For other properties of this special function, see \cite{Slater60}.

Note that, when $\alpha=-\frac{n}{2}$, we recover the profile found in the previous section, namely $v(r)=e^\frac{r^2}{4}$. In particular, when $\alpha_1=-\frac{n}{2}-1$ and $\alpha_2=-\frac{n}{2}-2$ in \eqref{eq:profilesolution}, we obtain two explicit solutions
\[
v_{\alpha_1}(r)=\left(1+\frac{r^2}{2n}\right)e^\frac{r^2}{4},\; v_{\alpha_2}(r)=\left(1+\frac{r^2}{n}+\frac{r^4}{4n(n+2)}\right)e^\frac{r^2}{4}.
\]
Therefore, from \eqref{eq:profilesolution}, we obtain a family of self-similar solutions of type II for the heat equation that blows up in finite time, namely
\begin{equation}\label{eq:familyalpha}
u_{\alpha}(t,x)=a(T-t)^\alpha\sum_{s=0}^{\infty}\frac{(-\alpha)_s}{(\frac{n}{2})_s} \frac{\left(\frac{\abs{x}^2}{4(T-t)}\right)^s}{s!},
\end{equation}
where  $a>0$ is an arbitrary constant and $\alpha\in (-\infty,-\frac{n}{2}]$. 

Collecting all the information, we give two explicit self-similar solutions of type II for the heat equation that blow-up in finite time $T$:
\begin{eqnarray*}
u_{\alpha_1}(t,x)&=&a(T-t)^{-\frac{n+2}{2}}\left(1+\frac{\abs{x}^2}{2n(T-t)}\right)e^\frac{\abs{x}^2}{4(T-t)},\\
u_{\alpha_2}(t,x)&=&a(T-t)^{-\frac{n+4}{2}}\left(1+\frac{\abs{x}^2}{n(T-t)}+\frac{\abs{x}^4}{4n(n+2)(T-t)^2}\right)e^\frac{\abs{x}^2}{4(T-t)},
\end{eqnarray*}
where  $a>0$ is an arbitrary constant. As far as we know, the family in \eqref{eq:familyalpha} is new. 

Finally, if we put in \eqref{eq:familyalpha} $\alpha_m=-\frac{n}{2}-m$, with $m=0,1,2,...$, we obtain an infinite family of explicit self-similar solutions of type II for the heat equation that blow up in finite time $T$, namely
\begin{equation*}
u_{\alpha_m}(t,x)=a(T-t)^{-\frac{n+2m}{2}}\sum_{s=0}^{\infty}\frac{(\frac{n}{2}+m)_s}{(\frac{n}{2})_s} \frac{\left(\frac{\abs{x}^2}{4(T-t)}\right)^s}{s!},
\end{equation*}
where  $a>0$ is an arbitrary constant. Note that when $m=0$, we recover the solution obtained in the previous section. 

\subsection*{Acknowledgements}
The author has been supported by ANID Fondecyt Grant Number 1221928, Chile.

\section*{Appendix: Existence of a local solution}
In this appendix we prove the existence of a local solution $v$ of equation \eqref{eq:goveq} which satisfy $v(0)=a>0$. The procedure via fixed point arguments is standard, but for the sake of completeness we will give the corresponding proof. For $\alpha\neq 0$, we need to consider two cases, according to the sign of $\alpha$. We only consider $\alpha, \beta>0$, the other cases being similar.
Note that, for $\alpha>0$, the equation \eqref{eq:goveq} is equivalent to the integral equation
\begin{equation}\label{eq:vfixedpoint}
v(r)=a-\int_{0}^{r}G(F(v)(s))ds,
\end{equation}
where 
\begin{equation}\label{eq:defG}
G(s)=\left(\frac{\alpha s}{c_{n,k}}\right)^\frac{1}{k},\; s\geq 0
\end{equation}
and 
\begin{equation}\label{eq:defF}
F(v)(s)=s^k\left(\frac{\beta}{\alpha}v(s)+\left(1-\frac{n\beta}{\alpha}\right)s^{-n}\int_{0}^{s}\tau^{n-1}v(\tau)d\tau\right).
\end{equation}
Now we consider, for $a>\delta>0,\; B_\delta(a):=\{\varphi\in C([0,\rho]):\sup\{\abs{\varphi(s)-a}:s\in [0,\rho]\}<\delta\}$. Here $\rho$ and $\delta$ are to be chosen later.

Now we define
\[
\mathcal{J}(\varphi)(r)=a-\int_{0}^{r}G(F(\varphi)(s))ds.
\]
We find a solution of the equation $\varphi=\mathcal{J}(\varphi)$ in $B_\delta(a)$ where $\rho$ and $\delta$ are small positive numbers which will be chosen sufficiently close to zero.
Obviously $\mathcal{J}(\varphi)\in C([0,\rho])$, and from the definition of $B_\delta(a),\; \varphi(r)\in [a-\delta,a+\delta]$ for all $r\in [0,\rho]$. Moreover, simple calculations show that, for small $\delta$, $F(\varphi)$ is positive on $[0,\rho]$ for all $\varphi\in B_\delta(a)$. More precisely, we have
\begin{equation}\label{eq:Flowerbound}
F(\varphi)(s)\geq As^k,\; \mbox{ for all}\; s\in [0,\rho],
\end{equation}
where
\[
A=\begin{cases}
\frac{a^2}{2n}\;\; \mbox{if}\,\; 0<a<1,\\
\frac{1}{2n}\;\; \mbox{if}\,\; a=1,\\
\frac{a}{2n}\;\; \mbox{if}\,\; a>1.
\end{cases}
\]
To verify \eqref{eq:Flowerbound}, take $\delta$ small enough so that
\[
0<\delta<
\begin{cases}
\frac{a}{2}\;\; \mbox{if}\,\; \frac{n\beta}{\alpha}<1,\\
\frac{a}{2}\frac{1}{2\frac{n\beta}{\alpha}-1}\;\; \mbox{if}\,\; \frac{n\beta}{\alpha}\geq 1.
\end{cases}
\]
Taking into account that the function $r\rightarrow \frac{G(r)}{r}$ is decreasing on $(0,\infty)$, we have
\begin{eqnarray*}
\abs{\mathcal{J}(\varphi)(r)-a}&\leq&\int_{0}^{r}\frac{G(F(\varphi)(s))}{F(\varphi)(s)}\abs{F(\varphi)(s)}ds\\
&\leq&\int_{0}^{r}\frac{G(As^k)}{As^k}\abs{F(\varphi)(s)}ds
\end{eqnarray*}
for $r\in [0,\rho]$. 

On the other hand,
\[
\abs{F(\varphi)(s)}\leq Cs^k,\;\; \mbox{where}\;\; C=\left(\frac{\beta}{\alpha}+\abs{1-\frac{n\beta}{\alpha}}\right)(a+\delta)>0.
\]
We thus get
\[
\abs{\mathcal{J}(\varphi)(r)-a}\leq\frac{C}{2A}\left(\frac{A\alpha}{c_{n,k}}\right)^{\frac{1}{k}}r^2
\]
for all $r\in [0,\rho]$. Choose $\rho$ small enough so that
\[
\abs{\mathcal{J}(\varphi)(r)-a}\leq\delta,\;\; \varphi\in B_\delta(a).
\]
Hence $\mathcal{J}(\varphi)\in B_\delta(a)$.

Next we show that $\mathcal{J}$ is a contraction in some interval $[0,\tilde{r}]$, where $\tilde{r}=\tilde{r}(a)$. Recall that, if $\tilde{r}$ is small enough, the ball $B_\delta(a)$ is invariant under $\mathcal{J}$, i.e., $\mathcal{J}(B_\delta(a))\subset B_\delta(a)$. For such $\tilde{r}$ and any pair $\varphi,\,\psi\in B_\delta(a)$, we have
\begin{equation}\label{eq:difer}
\abs{\mathcal{J}(\varphi)(r)-\mathcal{J}(\psi)(r)}\leq\int_0^r\abs{G(F(\varphi)(s))-G(F(\psi)(s))}ds
\end{equation}
where $F(\varphi)$ is given by \eqref{eq:defF}. Now, define
\[
H(s)=\min\{F(\varphi)(s),F(\psi)(s)\}.
\]
As a consequence of estimate \eqref{eq:Flowerbound}, we have
\[
H(s)\geq As^k\;\; \mbox{for}\;\; 0\leq s\leq r\leq\tilde{r},
\]
whence
\begin{equation}\label{eq:diferestimate}
\begin{split}
\abs{G(F(\varphi)(s))-G(F(\psi)(s))}&\leq\abs{\frac{G(H(s))}{H(s)}(F(\varphi)(s)-F(\psi)(s))}\\
&\leq\frac{G(As^k)}{As^k}\abs{F(\varphi)(s)-F(\psi)(s)}.
\end{split}
\end{equation}
Moreover,
\begin{equation}\label{eq:Linftyestimate}
\abs{F(\varphi)(s)-F(\psi)(s)}\leq\tilde{C}\norm{\varphi-\psi}_{\infty}s^k ,
\end{equation}
where 
\[
\tilde{C}=\frac{\beta}{\alpha}+\abs{\frac{1}{n}-\frac{\beta}{\alpha}}.
\]
Combining \eqref{eq:difer}, \eqref{eq:diferestimate} and \eqref{eq:Linftyestimate}, we obtain
\[
\abs{\mathcal{J}(\varphi)(r)-\mathcal{J}(\psi)(r)}\leq\frac{\tilde{C}}{2A}\left(\frac{A\alpha}{c_{n,k}}\right)^{\frac{1}{k}}r^2\norm{\varphi-\psi}_{\infty}
\]
for any $r\in [0,\tilde{r}]$. Choosing $\tilde{r}$ small enough, $\mathcal{J}$ is a contraction. The Banach fixed point theorem now implies the existence of a unique fixed point of $\mathcal{J}$ in $B_\delta(a)$, which is a solution of \eqref{eq:vfixedpoint} and, consequently, of \eqref{eq:goveq} and \eqref{eq:inicond}. As usual, this solution can be extended to a maximal interval $[0,R_0),\; 0<R_0\leq\infty$.
\bibliographystyle{plain}
\bibliographystyle{apalike}
\bibliography{kHessianbib}
\end{document}